\documentclass[reqno]{amsart}

\usepackage{hyperref}
\usepackage[all]{xy}

\numberwithin{equation}{section}

\newtheorem{thm}[equation]{Theorem}
\newtheorem{cor}[equation]{Corollary}
\newtheorem{lem}[equation]{Lemma}

\theoremstyle{remark}

\newtheorem{rem}[equation]{Remark}

\newcommand{\bQ}{\mathbb{Q}}
\newcommand{\bR}{\mathbb{R}}

\newcommand{\bZ}{\mathbb{Z}}
\newcommand{\cC}{\mathcal{C}}
\newcommand{\Diff}{\mathrm{Diff}}
\newcommand{\GL}{\mathrm{GL}}
\newcommand{\hofib}{\mathrm{hofib}}
\newcommand{\Homeo}{\mathrm{Homeo}}
\newcommand{\Hom}{\mathrm{Hom}}
\newcommand{\id}{\mathrm{id}}
\newcommand{\Torr}{\mathrm{Tor}}
\newcommand{\trf}{\mathrm{trf}}
\newcommand{\inter}[1]{\mathrm{int}(#1)}
\newcommand{\Cob}{\mathrm{Cob}}
\newcommand{\hocolim}{\mathrm{hocolim}}
\newcommand{\Fr}{\mathrm{Fr}}
\newcommand{\Top}{\mathrm{Top}}
\newcommand{\rp}{\mathbb{R}\mathbb{P}}

\title{Diffeomorphisms of odd-dimensional discs, glued into a manifold}

\author{Johannes Ebert}
\email{johannes.ebert@uni-muenster.de}
\address{
Mathematisches Institut\\
WWU M{\"u}nster\\
Einsteinstr. 62\\
48149 M{\"u}nster\\
Germany
}

\date{\today}

\begin{document}

\begin{abstract}
For a compact $(2n+1)$-dimensional smooth manifold, let $\mu_M : B \Diff_\partial (D^{2n+1}) \to B \Diff (M)$ be the map that is defined by extending diffeomorphisms on an embedded disc by the identity. 

By a classical result of Farrell and Hsiang, the rational homotopy groups and the rational homology of $ B \Diff_\partial (D^{2n+1})$ are known in the concordance stable range. 

We prove two results on the behavior of the map $\mu_M$ in the concordance stable range. Firstly, it is \emph{injective} on rational homotopy groups, and secondly, it is \emph{trivial} on rational homology, if $M$ contains sufficiently many embedded copies of $S^n\times S^{n+1} \setminus \mathrm{int}(D^{2n+1})$. We also show that $\mu_M$ is generally not injective on homotopy groups outside the stable range. 

The homotopical statement is probably not new and follows from the theory of smooth torsion invariants. The noninjectivity outside the stable range is based on recent work by Krannich and Randal-Williams. The homological statement relies on work by Botvinnik and Perlmutter on diffeomorphism of odd-dimensional manifolds. 
\end{abstract}

\maketitle

\section{Introduction}

For a smooth compact manifold $M$ with boundary, we denote by $\Diff (M)$ the topological group of diffeomorphisms of $M$ and by $\Diff_\partial (M) \subset \Diff (M)$ the subgroup of those diffeomorphisms which agree with the identity near $\partial M$. 
A celebrated classical result by Farrell-Hsiang \cite{FarrellHsiang} states that 
\begin{equation}\label{eqn:farrellhsiang}
\pi_k (B\Diff_\partial (D^{2n+1}))\otimes \bQ \cong 
\begin{cases}
\bQ & k \equiv 0 \pmod 4\\
0 & k \not \equiv 0 \pmod 4
\end{cases}
\end{equation}
in a range of degrees which was originally given by $k<\frac{2n+1}{6}-6$; but \eqref{eqn:farrellhsiang} holds more generally if $k \leq \phi^\bQ(D^{2n})$, where $\phi^\bQ(D^{2n})$ is the \emph{rational concordance stable range for $D^{2n}$} which we briefly recall. 

For a compact smooth manifold $M$, let 
\[
C(M):= \Diff (M\times [0,1], M \times \{0\} \cup \partial M \times [0,1])
\]
denote the concordance diffeomorphism group of $M$, and let $\sigma_+: C(M) \to C(M \times [0,1])$ be the (positive) suspension map defined in e.g. \cite[\S 6.2]{IgusaBook}. Define $\phi(M)$ to be the largest integer $k$ so that the maps $\sigma_+: C(M\times [0,1]^m) \to C(M\times [0,1]^{m+1})$, $m \geq 0$, are all $k$-connected. Similarly, one defines $\phi^\bQ(M)\geq \phi(M)$ using rational connectivity instead of connectivity (this makes sense if $\dim (M) \geq 6$ as $\pi_0 (C(M))$ is abelian in that case, by \cite[Lemma 1.1]{HatcherWagoner}).

Igusa's stability theorem \cite[p.6]{IgusaStability} states that 
\begin{equation}\label{eqn:igusarange}
\phi(M^d) \geq \min(\frac{d-7}{2},\frac{d-4}{3}).
\end{equation}
Recent work by Krannich and Randal-Williams \cite{KrannRW} gives the optimal range in which \eqref{eqn:farrellhsiang} holds. Corollary B of \cite{KrannRW} shows that 
\begin{equation}\label{eqn:optimalranmge}
\phi^\bQ(D^d)=d-4 \; \text{if} \; d \geq 10,
\end{equation}
and hence \eqref{eqn:farrellhsiang} holds if $k \leq 2n-4$, provided that $n \geq 5$, and Theorem A of \cite{KrannRW} improves this slightly to $k \leq 2n-3$, again for $n \geq 5$. These results slightly exceed \cite[Corollary B]{Krannich}. 

For an arbitrary smooth compact and nonempty manifold $M$ of dimension $2n+1$, choose an embedding $D^{2n+1} \to \inter{M}$. Extending diffeomorphisms by the identity gives a gluing map
\[
\mu_M^\partial: \Diff_\partial (D^{2n+1}) \to \Diff_\partial (M); 
\]
we may also consider the composition 
\[
\mu_M: \Diff_\partial (D^{2n+1}) \stackrel{\mu_M^\partial}{\to} \Diff_\partial (M) \to  \Diff (M).
\]
The purpose of this note is to study the effect of the maps $B \mu_M$ and $B\mu_M^\partial$ on rational homotopy and homology. The precise choice of the embedding does not play a role for this question as long as $M$ is connected. This is because the homotopy class of $B \mu_M$ only depends on the isotopy class of the embedding. If $M$ is connected and not orientable, there is only one isotopy class of embeddings, and if $M$ is connected and orientable, there are two such isotopy classes which differ by the reflection automorphism of the group $\Diff_\partial (D^{2n+1})$. 
\begin{thm}[Homotopical theorem]\label{mainthm:homotopy}
For every $(2n+1)$-dimensional manifold $M$, the maps 
\[
(\mu_M )_*: \pi_k (B\Diff_\partial (D^{2n+1}))\otimes \bQ \to \pi_k (B\Diff (M))\otimes \bQ
\]
and 
\[
(\mu_M^\partial )_*: \pi_k (B\Diff_\partial (D^{2n+1}))\otimes \bQ \to \pi_k (B\Diff_\partial (M))\otimes \bQ
\]
are injective when $k \neq 1$ and $k \leq \phi^\bQ(D^{2n})$. 
\end{thm}
\begin{rem}
Note that $\pi_1 (B\Diff_\partial (D^{2n+1}))=\pi_0 (\Diff_\partial (D^{2n+1}))$ is always a finite group; this is trivial when $n=0$ and follows from \cite{Cerf} for $n=1$. For $n \geq 2$, \cite[Corollaire 2]{CerfIsotopy} and the h-cobordism theorem identifies $\pi_0 (\Diff_\partial (D^{2n+1}))$ with the group of homotopy $(2n+2)$-spheres which is finite by \cite{KervMiln}. 
\end{rem}
\begin{rem}\label{rem:close-to-threshold}
By \cite[Corollary B]{KrannRW}, Theorem \ref{mainthm:homotopy} holds for $k \leq 2n-4$ if $n \geq 5$. Theorem \ref{mainthm:homotopy} is also true for $k = 2n-3$ and $n \geq 5$, since $\pi_{2n-3}(B \Diff_\partial (D^{2n+1})\otimes \bQ=0$ for such $n$ by \cite[Theorem A]{KrannRW}.
\end{rem}
Theorem \ref{mainthm:homotopy} could have been proven with little effort in \cite{BDKW} at latest. It was in fact known by experts and we learnt the statement from Mauricio Bustamante. The proof is given here for sake of completeness and to contrast it with our main result (Theorem \ref{maintheoremhomology} below), which seemingly goes into the opposite direction. 

It concerns the effect of $\mu_M$ in rational homology. Since $B \Diff_\partial (D^{2n+1})$ is a connected $E_{2n+1}$-space (and hence a homotopy commutative H-space), $H_* (B \Diff_\partial (D^{2n+1});\bQ)$ (with the Pontrjagin product) is the free graded-commutative algebra generated by $\pi_* (B \Diff_\partial (D^{2n+1})) \otimes\bQ$. Therefore in the concordance stable range, $H_* (B \Diff_\partial (D^{2n+1});\bQ)$ is a polynomial algebra with one generator in each dimension divisible by $4$. 

Let
\[
U_g^n := \sharp^g (S^n \times S^{n+1})
\]
be the connected sum of $g$ copies of $S^n \times S^{n+1}$, and let 
\[
U^n_{g,1}:= U_g^n \setminus \inter{D^{2n+1}}
\]
be $U_g^n$ with the interior of a disc removed.

\begin{thm}[Homological theorem]\label{maintheoremhomology}
Let $M$ be a connected manifold of dimension $2n+1 \geq 9$ and suppose that $M$ contains an embedded copy of $U_{g,1}^n$. Then the maps 
\[
(\mu_M)_*: \tilde{H}_k (B \Diff_\partial (D^{2n+1});\bQ) \to \tilde{H}_k (B \Diff (M);\bQ)
\]
and 
\[
(\mu_M^\partial)_*: \tilde{H}_k (B \Diff_\partial (D^{2n+1});\bQ) \to \tilde{H}_k (B \Diff_\partial (M);\bQ)
\]
are trivial if $k \leq \phi^\bQ(D^{2n})+1$ and $k \leq \frac{g-4}{2}$.
\end{thm}

Finally, using the recent work \cite{KrannRW}, we can show that the range for the validity of Theorem \ref{mainthm:homotopy} given in Remark \ref{rem:close-to-threshold} is optimal.

\begin{thm}\label{thm:optimality}
For even $n \geq 6$, there is a closed $(2n+1)$-dimensional smooth manifold $M$ such that the kernel of 
\[
(\mu_M)_*: \pi_{2n-2}(B \Diff_\partial (D^{2n+1})) \otimes \bQ \to \pi_{2n-2}(B \Diff (M)) \otimes \bQ
\]
is nonzero.
\end{thm}

\subsection*{Acknowledgements}

I am grateful to Mauricio Bustamante for stating Theorem \ref{mainthm:homotopy} as a fact to me, to Markus Land and Oscar Randal--Williams for some remarks, in particular concerning the proof of Theorem \ref{cor:frombotperl} and to Manuel Krannich for his detailed feedback on the entire paper. 

I am particularly indebted to the anonymous referee for his/her exceptionally speedy and thorough report. In particular, he/she suggested to improve the main results by using \cite{KrannRW}, and Theorem \ref{thm:optimality} together with its proof.

I was supported by the Deutsche Forschungsgemeinschaft (DFG, German Research Foundation) -- Project-ID 427320536 -- SFB 1442, as well as under Germany’s Excellence Strategy EXC 2044 -- 390685587, Mathematics M\"unster: Dynamics–Geometry–Structure.

\section{Proof of the homotopical theorem}

The proof of Theorem \ref{mainthm:homotopy} relies on higher torsion invariants as axiomatized by Igusa \cite{IgusaTorsion}, and we review some background beforehand.
Let $K(\bZ)$ be the algebraic $K$-theory spectrum of $\bZ$ and let 
\[
u: Q(S^0) \to \Omega^\infty K(\bZ)
\]
be the unit map on infinite loop spaces. 

Consider a fibration $\pi:E \to B$ with fibre $M$, a finite CW complex of dimension $d$. Let 
\begin{equation}\label{eqn:monodromy}
\rho_k(\pi) : B \to B \GL (H_k (M;\bZ))
\end{equation}
be the map induced by the monodromy action of the fundamental group on the homology of the fibre. The $\bZ$-module $H_k (M;\bZ)$ is finitely generated, and hence there is a canonical map 
\[
\iota: B \GL(H_k (M;\bZ))\to \Omega^\infty K(\bZ);
\]
such a map exists even if the homology groups are not free, essentially because $\bZ$ is a regular ring (each finitely generated $\bZ$-module has a finite length resolution by projective finitely generated $\bZ$-modules). See the discussion leading up to \cite[Proposition 6.7]{DWW} for details. 

The map $\iota$ hits the component of $[H_k (M;\bZ)] \in K_0 (\bZ)= \pi_0 (\Omega^\infty K(\bZ))\cong \bZ$. 
The algebraic $K$-theory Euler characteristic of the fibration $\pi$ is the alternating sum 
\[
\chi(\pi):= \sum_{k=0}^d (-1)^k \iota \circ \rho_k(\pi): B \to \Omega^\infty K(\bZ),
\]
where $d=\dim (M)$ and we have used the $H$-space structure on $\Omega^\infty K(\bZ)$ to form the sum. Of course, $\chi(\pi)$ hits the component indexed by $ \chi(M) \in \bZ=K_0 (\bZ)$. 

The fibration $\pi$ has an associated transfer map \cite{BeckGott2}
\[
\trf_\pi: \Sigma^\infty B_+ \to \Sigma^\infty E_+
\]
on the level of suspension spectra; we mostly consider its adjoint, also denoted
\[
\trf_\pi: B \to Q(E_+). 
\]

The Dwyer--Weiss--Williams index theorem \cite[Corollary 8.12]{DWW} implies that if $\pi$ is a \emph{smooth} fibre bundle, the diagram 
\begin{equation}\label{eq:dwwdiagram}
\xymatrix{
B \ar[dr]_{\chi(\pi)} \ar[r]^-{\trf_\pi} & Q(E_+) \ar[d]^{u\circ c}\\
    & \Omega^\infty K(\bZ)
}
\end{equation}
commutes up to a preferred homotopy (here $u\circ c$ is the composition of the unit map with the collapse map $c: Q(E_+) \to Q(S^0)$). 

\begin{rem}
Actually, \cite[Theorem 8.5]{DWW} proves a stronger version involving the algebraic $K$-theory $A(E)$ of the space $E$ and a fibrewise version thereof. Raptis and Steimle gave a substantially simpler proof of the homotopy-commutativity of \eqref{eq:dwwdiagram} in \cite{RS1}; they also showed \cite[Theorem 8.5]{DWW} for smooth bundles in \cite{RS2}. 
\end{rem}

The diagram \eqref{eq:dwwdiagram} can be used to define secondary invariants under additional hypotheses on the bundle $\pi$; we follow the approach of \cite{BadziochDorabialaWilliams} here, with some modifications. The extra hypothesis to be made is that the monodromy action of $\pi_1(B)$ on $H_k (M;\bZ)$ is unipotent for all $k$ (in \cite{BadziochDorabialaWilliams}, the authors consider homology with coefficients in a field, but the construction generalizes to regular rings such as $\bZ$, see \cite[Remark 6.11]{BadziochDorabialaWilliams}). Moreover, we assume as in \cite{BadziochDorabialaWilliams} that the base space $B$ is a compact manifold, possibly with boundary. 

Under these assumptions, the map $\chi(\pi)$ comes with a preferred homotopy to the constant map to the point 
\[
 \chi(M):=\sum_{k=0}^d (-1)^k [H_k (M;\bZ)] \in \Omega^\infty K(\bZ),
\]
see \cite[Theorem 6.7]{BadziochDorabialaWilliams}. Combining this homotopy with the preferred homotopy from \eqref{eq:dwwdiagram} yields a map 
\begin{equation}\label{eqn:definitiontaumap}
T (\pi): B \to \hofib_{\chi(M)} (u) \simeq \hofib_0 (u)
\end{equation}
(use the infinite loop space structures to identify the homotopy fibres). 
Using Borel's computation \cite{Borel} of $\pi_* (\Omega^\infty K(\bZ)) \otimes \bR$, one can define characteristic classes of unipotent smooth bundles as follows. Firstly 
\begin{equation}\label{eqn:borels-calculation0}
H^* (\hofib_0 (u);\bR)= \bR [a_4, a_8, \ldots]
\end{equation}
for certain generators $a_{4k}$ of degree $4k$ (the transgression of $a_k$ is the Borel class in $ H^{4k+1}(\Omega^\infty K(\bZ);\bR)$). 
Following \cite[\S 7]{BadziochDorabialaWilliams} (but using the notation of \cite{BDKW}), one defines
\begin{equation}\label{eq:torsionclasscohomology}
t_{4k}^s(\pi):= T(\pi)^* a_{4k} \in H^{4k}(B;\bR).
\end{equation}
It is convenient for us to replace the coefficient field by $\bQ$, which can be done as follows. Firstly, the Borel class comes from a spectrum cohomology class $b_k \in H^{4k+1} (K(\bZ);\bR)$. Secondly, Borel showed that $H^{4k+1} (K(\bZ);\bR)$ is $1$-dimensional, and so there is $\alpha_k \in \bR^\times$ such that $\alpha_k b_k$ lies in $H^* (K(\bZ);\bQ)$. We now define 
\begin{equation}\label{eq:torsionclasscohomology2}
\overline{t}_{4k}^s(\pi)=\overline{t}_{4k}^s(E):= \alpha_k t_{4k}^s(\pi) \in H^{4k}(B;\bQ). 
\end{equation}
The construction of \eqref{eqn:definitiontaumap} is given in \cite{BadziochDorabialaWilliams} only for compact manifold bases; the definition of \eqref{eq:torsionclasscohomology2} can be generalized to arbitrary base spaces as follows. For an arbitrary unipotent bundle $E \to B$, we define $\overline{t}_{4k}^s (E) \in H^{4k}(B;\bQ)$ as the class corresponding to the homomorphism 
\begin{equation}\label{eqn:framedbordismhom}
\Omega_{4k}^{\mathrm{fr}}(B) \to \bQ; \; [X,f] \mapsto \langle \overline{t}_{4k}^s (f^* E);[X] \rangle 
\end{equation}
from the framed bordism group of $B$ under the isomorphism $\Hom (\Omega^{\mathrm{fr}}_{4k}(B);\bQ) \cong H^{4k}(B;\bQ)$ (one needs \cite[Proposition 7.3]{BadziochDorabialaWilliams} to show that \eqref{eqn:framedbordismhom} is well-defined).

Now let $\Torr (M) \subset \Diff (M)$ be the \emph{Torelli diffeomorphism group}, i.e. the subgroup of those diffeomorphisms which act as the identity on $H_* (M;\bZ)$ and $H_* (M,\partial M;\bZ)$. The universal $M$-bundle over $B  \Torr (M)$ is clearly unipotent, and combining the classes $\overline{t}_{4k}^s$, we obtain a map 
\begin{equation}\label{eqn:torsion-torellispace}
\tau_M : B \Torr (M) \to \prod_{k \geq 1} K(\bQ,4k). 
\end{equation}
Observe that $\Torr (D^{2n+1})=\Diff^+ (D^{2n+1})$ is the group of orientation-preserving diffeomorphisms, we obtain in particular
\begin{equation}\label{smoothtoriosndisc}
\tau_{D^{2n+1}}: B \Diff_\partial(D^{2n+1})  \to B \Diff^+ (D^{2n+1}) = B \Torr (D^{2n+1}) \to \prod_{k \geq 1} K(\bQ,4k)
\end{equation}

Farrell-Hsiang's theorem might be restated as follows.

\begin{thm}\label{thm:farrellhsiang}
The map $\tau_{D^{2n+1}}$ induces an isomorphism on rational homotopy groups in degrees at most $\phi^\bQ(D^{2n})$.
\end{thm}

\begin{proof}
The map $\tau_{D^{2n+1}}$ factors through $B \Diff_\partial (D^{2n+1}) \to BC(D^{2n}) \to B \Diff^+ (D^{2n+1})$. Consider the diagram 
\begin{equation}\label{eqn:productdiagram}
\xymatrix{
B C(D^{2n}) \ar[d]^{\sigma_+^{\circ m}} \ar[r] & B \Diff^+ (D^{2n+1}) \ar[d]^{\times \id_{D^{m}}} \\
 B C(D^{2n+m}) \ar[r] & B \Diff^+ (D^{2n+m+1});
}
\end{equation}
the left vertical map is a composition of the suspension map and the right vertical map is given by taking products with $D^m$ and an identification $D^{2n+m+1}=D^{2n+1} \times D^m$. The square commutes up to homotopy by the definition of the suspension map $\sigma_+$.
A special case of \cite[Theorem 7.1]{BDKW} states that for each unipotent bundle $E \to B$, $\overline{t}_{4k}^s (E \times D^{m})=\overline{t}_{4k}^s (E)$. It follows that $\tau_{D^{2n+1}}: B \Diff_\partial (D^{2n+1}) \to \prod_{k \geq 1} K(\bQ,4k)$ factors as 
\begin{equation}\label{eqn:factoring-torsionmap}
B \Diff_\partial (D^{2n+1}) \to BC(D^{2n}) \to B \cC(D^{2n}):= \hocolim_m B C(D^{2n+m}) \to \prod_{m \geq 1} K(\bQ,4k).
\end{equation}
All three maps in \eqref{eqn:factoring-torsionmap} induce isomorphisms on rational homotopy up to degree $\phi^\bQ(D^{2n})$. This is true for the second map by definition. 

The third map is a rational equivalence. Since the Whitehead group of $\pi_1 (D^{2n})$ is trivial, $B \cC(D^{2n})$ is the stable $h$-cobordism space $\mathcal{H}(D^{2n})$. The stable $h$-cobordism theorem \cite{WJR} states an equivalence $\mathcal{H}(D^{2n}) \simeq \hofib (Q(S^0) \to A(*))$, and the linearization map from $A(*)$ to $\Omega^\infty K(\bZ)$ induces a rational equivalence 
\[
\hofib (Q(S^0) \to A(*)) \to \hofib (Q(S^0) \to \Omega^\infty K(\bZ)). 
\]
Together with \cite{Borel}, this shows that the rational homotopy groups of $B \cC(D^{2n})$ and of $\prod_{m \geq 1} K(\bQ,4k)$ are abstractly isomorphic and at most $1$-dimensional. 

To conclude that the third map in \eqref{eqn:factoring-torsionmap} is a rational isomorphism, it is therefore enough to prove that the induced map on rational homotopy groups is nontrivial whenever its target is nonzero, and this amounts to proving that for each $k \geq 1$, there is $m $ and an element in $\pi_{4k} (BC(D^{2n+m}))$ such that $\overline{t}_{D^{2n+m+1}}^s$ is nontrivial on that element. 
This was done by Igusa in \cite[Theorem 6.4.2]{IgusaBook}, but with the \emph{higher Franz--Reidemeister torsion} classes $t_{4k}^{FR} \in H^{4k}(BC(D^{2n+m});\bR)$ in place of $\overline{t}_{4k}^s$. These were constructed using ideas from Morse theory in \cite[\S 5.7.2]{IgusaBook} (for bundles with structure group $\Torr (M)$) and in \cite[\S 2.11]{IgusaFraming} for unipotent bundles. The main theorem of \cite{BDKW} shows that there is a universal constant $\lambda_{4k} \in \bR^\times$ such that $\overline{t}_{4k}^s (\pi)= \lambda_{4k}t_{4k}^{IK} (\pi)$ for all unipotent bundles over compact manifold bases, and so the third map in \eqref{eqn:factoring-torsionmap} is a rational equivalence. 

It is shown in \cite[\S 6.5]{IgusaBook} that the first map in \eqref{eqn:factoring-torsionmap} induces an isomorphism on rational homotopy groups up to degree $\phi^\bQ(D^{2n})$. In loc.cit., the result is stated in terms of \eqref{eqn:igusarange}, so we give a few more details here. One looks at the fibre sequence
\[
\Diff_\partial (D^{2n+1}) \to C(D^{2n}) \to  \Diff_\partial (D^{2n}).
\]
The two maps are compatible with the following involutions on the spaces: the group inversion on $\Diff_\partial (D^{2n})$, an involution defined at the beginning of \cite[\S 6.5]{IgusaBook} on $C(D^{2n})$, and the involution 
\begin{equation}\label{eqn:involutionondiscdiff}
I: \Diff_\partial (D^{2n+1}) \to \Diff_\partial (D^{2n+1})
\end{equation}
given by conjugation with the reflection map 
\[
r(x_1, \ldots, x_{2n+1}) := (x_1, \ldots, x_{2n}, -x_{2n+1}). 
\]
The rational homotopy sequence of the fibration splits into negative and positive eigenspaces of these involutions. An Eckmann-Hilton argument proves that $\pi^+_*(\Diff_\partial (D^{2n+1}))\otimes \bQ \to \pi^+_*(C(D^{2n})) \otimes \bQ$ is an isomorphism in all degrees. On the other hand, $I_*=\id$ on $\pi_k (\_)\otimes \bQ$ for $k \leq \phi^\bQ(D^{2n})$; Corollary 6.5.3 of \cite{IgusaBook} states this when $k$ is in the range given by \eqref{eqn:igusarange}, but the proof clearly works for $k \leq \phi^\bQ(D^{2n})$. Hence for $k \leq \phi^\bQ(D^{2n})+1$
\[
\pi_k (B \Diff_\partial (D^{2n+1}) )\otimes \bQ = \pi_k^+ (B \Diff_\partial (D^{2n+1}) )\otimes \bQ \cong \pi_k^+ (BC(D^{2n}))\otimes \bQ. 
\]
Finally, $\pi_k^-(BC(D^{2n}))\otimes \bQ=0$ for $k \leq \phi^\bQ(D^{2n})$. To see this, observe that the stabilization map $BC(D^d))\to BC(D^{d+1})$ switches the eigenspaces of the involutions by \cite[Lemma 6.5.1]{IgusaBook}. Hence it is enough to check that $\pi_k^-(BC(D^{2n}))\otimes \bQ=0$ for very large $n$, and this follows from Theorem 6.4.2 and Lemma 6.5.4 of \cite{IgusaBook}, using that the third map in \eqref{eqn:factoring-torsionmap} is a rational equivalence.
\end{proof}

\begin{proof}[Proof of Theorem \ref{mainthm:homotopy} for closed $M$]
We first consider the case where $M$ is closed. The cohomology classes $\overline{t}_{4k}^s$ have the following additivity property: for unipotent bundles $\pi_j: E_j \to B$ with common (vertical) boundary bundle $E_{01}=\partial E_0=\partial E_1$, the boundary bundle $\pi_{01}:E_{01} \to B$, as well as the glued bundle $\pi: E =E_0 \cup_{\partial_{E_j}} E_1 \to B$ are also unipotent, and we have 
\begin{equation}\label{additivitysmoothtorsion}
\overline{t}_{4k}^s  (\pi)+\overline{t}_{4k}^s  (\pi_{01})= \overline{t}_{4k}^s  (\pi_1)+ \overline{t}_{4k}^s  (\pi_2)\in H^{4k}(B;\bQ).
\end{equation}
This is proven for $B$ a compact manifold in \cite[Corollary 5.2]{BDKW}, the case of a general base follows by using framed cobordism as in the construction of $\overline{t}_{4k}^s$ for general base spaces. 

Next, the map $\mu_M: B \Diff_\partial (D^{2n+1}) \to B \Diff (M)$ lifts to $\tilde{\mu}_M: B \Diff_\partial (D^{2n+1}) \to B \Torr (M)$, and \eqref{additivitysmoothtorsion} shows that 
\begin{equation}\label{eqn:tau-mu-homotopy}
\tau_M  \circ \tilde{\mu}_M \sim \tau_{D^{2n+1}}. 
\end{equation}
By Theorem \ref{thm:farrellhsiang}, it follows that $\tilde{\mu}_M$ is injective on $\pi_k (\_)\otimes \bQ$ when $k \leq \phi^\bQ(D^{2n})$. 
Because $\Torr (M) \subset \Diff (M)$ is a union of path-components, $p_*:\pi_k (B \Torr (M)) \to \pi_k (B \Diff (M))$ is injective when $k=1$ and an isomorphism when $k \geq 2$, and $\mu_M = p \circ \tilde{\mu}_M$ is injective on rational homotopy groups up to degree $\phi^\bQ(D^{2n})$. 
\end{proof}

We have used that $M$ is closed in order to use \eqref{additivitysmoothtorsion} which in the quoted source is only covered for closed $M$. The case of a general $M$ reduces to the closed case by ``doubling''. Let $M$ be a manifold with boundary, let $A \subset \partial M$ be a compact codimension $0$ submanifold and form $M \cup_{\partial M - \inter{A}} M$. This is a manifold with boundary $A\cup_{\partial A} A$. 
Let $\Diff_A(M)$ be the group of diffeomorphisms which fix $A$ pointwise. There is a doubling map 
\[
d_{A}: B \Diff_A (M) \to B \Diff_\partial (M \cup_{\partial M - \inter{A}} M)
\]
given by extending a diffeomorphism with its reflection. The diagram
\begin{equation}\label{eqn:doublingdiagram}
\xymatrix{
B \Diff_\partial (D^{2n+1}) \ar[rr]^{\mu_M} \ar[d]^{d_{D^{2n}}} & & B \Diff_A (M) \ar[d]^{d_A} \\
B \Diff_\partial (D^{2n+1}\cup_{D^{2n}} D^{2n+1}) \ar[rr]^-{\mu_{M \cup_{\partial M - \inter{A}} M}} & & B \Diff_\partial (M \cup_{\partial M - \inter{A}} M)
}
\end{equation}
commutes up to homotopy, where $D^{2n} \subset \partial D^{2n+1}$ denotes a half-disc in the boundary.

\begin{lem}\label{lem:doubling-map-disc}
The doubling map induces an isomorphism 
\[
(d_{D^{2n}})_*: \pi_k (B \Diff_\partial (D^{2n+1}))\otimes \bQ \to \pi_k (B \Diff_\partial (D^{2n+1}\cup_{D^{2n}} D^{2n+1}))\otimes \bQ
\]
when $k \leq \phi^\bQ(D^{2n})+1$. 
\end{lem}

\begin{proof}
By the Eckmann-Hilton argument, the effect of $d$ on rational homotopy groups is the map
\[
1+(BI)_*: \pi_* (B \Diff_\partial (D^{2n+1})) \otimes \bQ\to \pi_* (B \Diff_\partial (D^{2n+1}))\otimes \bQ
\]
(here $I$ is the involution \eqref{eqn:involutionondiscdiff}, and we identified $D^{2n+1}\cup_{D^{2n}} D^{2n+1}=D^{2n+1}$). The lemma now follows from the fact that $BI_*=\id$ on $\pi_k (\_)\otimes \bQ$ for $k \leq \phi^\bQ(D^{2n})+1$ (see the proof of Theorem \ref{thm:farrellhsiang} above for more details). 
\end{proof}

\begin{rem}
The bound given in Lemma \ref{lem:doubling-map-disc} is optimal: \cite[Corollary 8.4]{KrannRW} shows that the involution acts nontrivially on $\pi_{2n-2}(B \Diff_\partial (D^{2n+1})) \otimes \bQ$ when $n  \geq 5$, while $\phi^\bQ(D^{2n})+1 = 2n-3$ for such $n$.
\end{rem}

\begin{proof}[Proof of Theorem \ref{mainthm:homotopy}, general case]
To prove the statement for $\mu_M$, use \eqref{eqn:doublingdiagram} with $A=\emptyset$ and apply Lemma \ref{lem:doubling-map-disc}. The statement for $\mu_M^\partial$ follows from that for $\mu_M$ in view of the definition of $\mu_M$. 
\end{proof}

\begin{rem}
From the proof of Theorem \ref{mainthm:homotopy} given above, one can also deduce a statement about $H_*(B \Torr (M);\bQ)$, namely that $\tilde{\mu}_M$ is injective on rational homology in the concordance stable range, at least when $M$ is closed or orientable. 

In the case where $M$ is closed, this follows from \eqref{eqn:tau-mu-homotopy} and Theorem \ref{thm:farrellhsiang}. For manifolds with nonempty boundary, a bit more care is needed to check that doubling really gives a map $\Torr (M) \to \Torr (M \cup_{\partial M}M )$. For oriented $M$, the argument goes as follows. 

Mapping both halves of $M \cup_{\partial M} M $ to $M$ and excision in homology gives a $\Diff(M)$-equivariant isomorphism $H_* (M \cup_{\partial M} M;\bZ) \cong H_* (M;\bZ) \oplus H_* (M,\partial M;\bZ)$, from which it follows that the double of $f$ also induces the identity on homology. 

We leave it to the reader to figure out statements in cohomology or a variant for nonorientable $M$.
\end{rem}

\section{Proof of the homological theorem}

We now turn to the proof of Theorem \ref{maintheoremhomology}, which relies on work by Botvinnik and Perlmutter \cite{PerlStab}, \cite{BotPerl}. To state their results, let 
\[
V_g^n := \natural^g (S^n \times D^{n+1})
\]
be the boundary connected sum of $g$ copies of $S^n \times D^{n+1}$, and let $D=D^{2n} \subset \partial V_g^n$ be a disk in the boundary of $V_g^n$. Note that $V_0^n = D^{2n+1}$. There is a stabilization map
\begin{equation}\label{eqn:stabilitypasf}
B \Diff_D (V_g^n) \to B \Diff_D (V_{g+1}^n),
\end{equation}
given by taking boundary connected sum with $S^n \times D^{n+1}$ at $D$ and extending diffeomorphisms by the identity. 
Perlmutter \cite[Theorem 1.1]{PerlStab} proved that the map \eqref{eqn:stabilitypasf} induces an isomorphism in homology in degrees $* \leq\frac{1}{2} (g-4)$, provided that $n \geq 4$. Botvinnik and Perlmutter \cite{BotPerl} computed the homology of $B \Diff_D (V_g^n)$ in the stable range. Let 
\[
\theta_n: BO(2n+1)\langle n \rangle\to BO(2n+1)
\]
be the $n$-connected cover of $BO(2n+1)$. Let $\pi:E \to B$ be a bundle with fibre $V_g^n$ and structure group $\Diff_D (V_g^n)$. The vertical tangent bundle $T_v E$ admits a $\theta_n$-structure, i.e. a bundle map $\ell:T_v E \to \theta_n^* \gamma_{2n+1}$ to the pullback of the universal bundle over $BO(2n+1)$. This $\theta_n$-structure is unique up to contractible choice, once the following condition is imposed. Inside $E$, there is a trivial $D$-subbundle $B \times D$. The restriction of the vertical tangent bundle $T_v E$ to $B \times D$ has a canonical trivialization, and one requires that $\ell$ is compatible with that trivialization (see \cite[Proposition 6.16]{BotPerl} for all this).

Let $\underline{\ell}: E \to BO(2n+1)\langle n \rangle$ be the map of spaces underlying $\ell$. Let
\[
\alpha_\pi: B \stackrel{\trf_\pi}{\to} Q(E_+) \stackrel{Q(\underline{\ell})}{\to} Q(BO(2n+1)\langle n \rangle_+)
\]
be the composition of the transfer with the map induced by $\underline{\ell}$. In particular, we can apply this construction to the universal bundle over $B \Diff_D (V_g^n)$ and obtain a map 
\begin{equation}\label{eqn:bitviinikperlmuttermap}
\alpha_g: B \Diff_D (V_g^n) \to Q_{1+ (-1)^n g} (BO(2n+1)\langle n \rangle_+);
\end{equation}
the target is the path component indexed by 
\[
\chi(V_g^n) = 1+ (-1)^n g \in \bZ = \pi_0 (Q(BO(2n+1)\langle n \rangle_+)). 
\]
The following is essentially \cite[Corollary B]{BotPerl}.
\begin{thm}[Botvinnik, Perlmutter]\label{thm:botperl}
Let $n \geq 4$. Then the map \eqref{eqn:bitviinikperlmuttermap} induces an isomorphism in integral homology in degrees $* \leq \frac{1}{2}(g-4)$.  
\end{thm}

Theorem \ref{thm:botperl} as stated above differs from the formulation given in \cite{BotPerl} in so far as loc.cit. does not mention the transfer at all, so some remarks have to be made here. For a fibration $\theta: X \to BO(d)$, Genauer \cite{Genauer} introduced the cobordism category $\Cob_\theta^\partial$ of $(d-1)$-dimensional $\theta$-manifolds with boundaries and their cobordisms (which have corners). He proved that there is a weak equivalence $B \Cob_\theta^\partial \simeq \Omega^{\infty-1} \Sigma^\infty X_+$, and the equivalence is given by a parametrized Pontrjagin-Thom construction (this result is parallel to the well-known \cite{GMTW} for the usual cobordism category). Given any bundle $\pi:E \to B$ of smooth compact $d$-manifolds with boundary equipped with a $\theta$-structure $\ell$ on the vertical tangent bundle, one obtains a tautological map $B \to \Omega B \Cob_\theta^\partial$, and from the description of the transfer for smooth bundles, one sees that the composition of this tautological map with Genauer's equivalence agrees with the composition $B \stackrel{\trf_\pi}{\to} Q(E_+) \stackrel{Q(\ell_+)}{\to} Q (X_+)$. Using this observation, one derives Theorem \ref{thm:botperl} from the results of \cite{BotPerl}. 

\begin{cor}\label{cor:frombotperl}
If $n \geq 4$, the iterated stabilization map 
\[
B \Diff_D (V_0^n) \to B \Diff_D (V_g^n)
\]
induces the zero map on integral reduced homology, in degrees $* \leq \frac{1}{2} (g-4)$. 
\end{cor}

\begin{proof}
The transfer has an additivity property \cite{BeckerSchultz} which implies that 
\[
\xymatrix{
B \Diff_D (V_0^n) \ar[r] \ar[d]^{\alpha_0} & B \Diff_D (V_g^n)\ar[d]^{\alpha_g}\\
Q_{1} (BO(2n+1)\langle n \rangle_+) \ar[r] & Q_{1+ (-1)^n g} (BO(2n+1)\langle n \rangle_+) 
}
\]
commutes up to homotopy; the lower map takes the sum with a fixed point in $Q_{(-1)^n g} (BO(2n+1)\langle n \rangle_+)$ and is a weak equivalence. We shall show that the left vertical map is trivial in reduced homology (in all degrees); this will imply the claim by Theorem \ref{thm:botperl}. 

The map $\alpha_0$ factors as 
\[
B \Diff_D (V_0^n) \stackrel{\trf}{\to} Q_1 ((E \Diff_D (V_0^n)\times_{\Diff_D (V_0^n)} V_0^n)_+) \stackrel{Q(\underline{l})}{\to}  Q_{1} (BO(2n+1)\langle n \rangle_+). 
\]
The map $Q(\underline{l})$ is induced from the vertical tangent bundle of the universal bundle $E \Diff_D (V_0^n)\times_{\Diff_D (V_0^n)} V_0^n \to B \Diff_D (V_0^n)$, which is trivial by the following argument: the choice of a point in $D \subset \partial V_0^n$ determines a section $s$ of the bundle, which is a homotopy equivalence as $V_0^n$ is a disc, and the pullback of the vertical tangent bundle along $s$ is trivial. Hence $Q(\underline{l})$ factors through $Q_1 (S^0)$ which is rationally acyclic. Therefore $\alpha_0$ induces the zero map on rational homology. 
This finishes the proof for rational homology, which is all we need for the proof of Theorem \ref{maintheoremhomology}. 

For the integral version, we use that the transfer is defined more generally for fibrations with finite CW fibres. It follows that there is a commutative diagram 
\[
\xymatrix{
B \Diff_D (V_0^n) \ar[r]^-{\trf} \ar[d] &  Q_1 ((E \Diff_D (V_0^n)\times_{\Diff_D (V_0^n)} V_0^n)_+) \ar[d]\\
B \Homeo_D (V_0^n) \ar[r]^-{\trf}  & Q_1 ((E \Homeo_D (V_0^n)\times_{\Homeo_D (V_0^n)} V_0^n)_+). 
}
\]
Because the map 
\[
\underline{\ell}: E \Diff_D (V_0^n)\times_{\Diff_D (V_0^n)} V_0^n \to B O(2n+1)\langle n \rangle
\]
is nullhomotopic as we just argued, it extends to a map 
\[
\underline{\ell}' : E \Homeo_D (V_0^n)\times_{\Homeo_D (V_0^n)} V_0^n \to B O(2n+1)\langle n \rangle.
\] 
Therefore $\alpha_0$ factors up to homotopy through $B \Homeo_D (V_0^n)$, which is contractible by the Alexander trick. 
\end{proof}

\begin{proof}[Proof of Theorem \ref{maintheoremhomology}]
Since both $\mu_M$ and $\mu_M^\partial$ factor through
\[
\mu_{U_{g,1}^n}^\partial:B \Diff_\partial (D^{2n+1}) \to B \Diff_\partial (U_{g,1}^n),
\]
it suffices to show that $\mu_{U_{g,1}^n}^\partial$ induces the trivial map in rational homology in the indicated range of degrees. Note that 
\[
V_g^n \cup_{\partial V_g^n \setminus \inter{D}} V_g^n = U_{g,1}^n. 
\]

Pick an embedding $f:D^{2n+1} \to V_0^n \subset V_g^n$ which is disjoint from the disc $D \subset V_0^n$ and such that $f^{-1} (\partial V_0^n)= D^{2n}$ is a disc in $\partial D^{2n+1}$. This gives a map $\mu:B \Diff_\partial (D^{2n+1}) \to B\Diff_D (V_0^n)$, and by Corollary \ref{cor:frombotperl}, the composition of that with the stabilization map to $B \Diff_D (V_g^n)$ is trivial in integral homology degrees $\leq \frac{1}{2} (g-4)$. 

Diagram \eqref{eqn:doublingdiagram} becomes 
\begin{equation}\label{eqn:doublingdiagram2}
\xymatrix{
B \Diff_\partial (D^{2n+1}) \ar[r] \ar[d]^{d} & B \Diff_D (V_g^n) \ar[d]^{d} \\
B \Diff_\partial (D^{2n+1}\cup_{D^{2n}} D^{2n+1}) \ar[r]^-{\mu} & B \Diff_\partial (U_{g,1}^n). 
}
\end{equation}
Lemma \ref{lem:doubling-map-disc} shows that the left vertical in \eqref{eqn:doublingdiagram2} induces an isomorphism in rational homotopy in degrees at most $\phi^\bQ(D^{2n})+1$. The same is true in rational homology, since both spaces are connected $H$-spaces and their rational homology is the free graded commutative algebra on the rational homotopy, and so the proof is complete.
\end{proof}

\section{Optimality of the range in the homotopical theorem}

\begin{proof}[Proof of Theorem \ref{thm:optimality}]
The composition $B \Diff_\partial (D^{2n+1}) \stackrel{\mu_M}{\to} B \Diff (M) \to B \Homeo (M)$ factors through $B \Homeo_\partial (D^{2n+1}) \simeq *$. Hence $\mu_M$ factors through the space
\[
\hofib (B \Diff(M) \to B \Homeo(M)). 
\]
By \cite[Theorem A]{KrannRW}, $\pi_{2n-2}(B \Diff_\partial (D^{2n+1})) \otimes \bQ \neq 0$ if $n \geq 6$. Therefore, it is enough to find a closed $(2n+1)$-manifold $M$ such that $\pi_{2n-2}(\hofib (B \Diff(M) \to B \Homeo(M))) \otimes \bQ =0$. 

Now by smoothing theory \cite[Essay V]{KirbSieb}, $\hofib (B \Diff(M) \to B \Homeo(M))$ is homotopy equivalent to a union of path components of the section space $\Gamma (M;\Fr(M)\times_{O(2n+1)} \frac{\Top(2n+1)}{O(2n+1)})$.

We now prove that 
\begin{equation}\label{eqn:vanishingRP2}
\pi_{2n-2} (\Gamma (\rp^2 \times S^{2n-1};\Fr(\rp^2 \times S^{2n-1})\times_{O(2n+1)} \frac{\Top(2n+1)}{O(2n+1)})) \otimes \bQ=0
\end{equation}
when $n$ is even. Homotopy groups of section spaces can be computed by means of the Federer spectral sequence \cite{Federer}; see \cite[\S 5.2]{KupRWTorAlg} for the variant we shall be using. Let $E \to B$ be a fibration over a finite-dimensional CW-complex with simply connected fibre $F$ and a fixed section $s$. Then there is a spectral sequence
\[
E_{p,q}^2 = H^{-p} (B;\pi_q(F)) \Rightarrow \pi_{p+q}(\Gamma(B;E),s)
\]
(the coefficient systems in the $E_2$-term are twisted). Hence $\pi_{m}(\Gamma(B;E),s)$ admits a finite filtration whose filtration quotients are subquotients of $H^{-p} (B;\pi_q(F))$ with $p+q=m$; and so in order to prove that $\pi_{m}(\Gamma(B;E),s) \otimes \bQ=0$, it suffices to show that $H^{-p} (B;\pi_q(F)) \otimes \bQ=0$ if $p+q=m$. 

Because $\frac{\Top(2n+1)}{O(2n+1)}$ is simply connected by \cite[V.5.0(4)--(5)]{KirbSieb}, we can apply the Federer spectral in the case at hand. Furthermore, by loc.cit., $\frac{\Top(2n+1)}{O(2n+1)}$ is rationally $(2n+2)$-connected. So the only entries in the $E_2$-page which could potentially be rationally nonzero and contribute to \eqref{eqn:vanishingRP2} are 
\[
H^{2n+1-i}(\rp^2 \times S^{2n-1};\pi_{4n-1-i}(\frac{\Top(2n+1)}{O(2n+1)})\otimes \bQ)
\]
for $0 \leq i \leq 2$. It remains to be shown that 
\begin{equation}\label{eqn:cohomology-rp2}
H^{2-i}(\rp^2 ;\pi_{4n-1-i}(\frac{\Top(2n+1)}{O(2n+1)})\otimes \bQ) =0
\end{equation}
when $0 \leq i \leq 2$. 

The fundamental group $\pi_1 (\rp^2)=C_2$ acts on $\pi_{4n-1-i}(\frac{\Top(2n+1)}{O(2n+1)})$ by conjugation with an isometry of determinant $-1$.
There is an isomorphism 
\[
\pi_{4n-1-i}(\frac{\Top(2n+1)}{O(2n+1)}) = \pi_{2n-2-i}(B \Diff_\partial (D^{2n+1}))
\]
coming from Morlet's theorem \cite[V.3.4]{KirbSieb} that $B \Diff_\partial (D^{2n+1}) \simeq \Omega_0^{2n+1} (\frac{\Top(2n+1)}{O(2n+1)})$. By the discussion in \cite[\S 8.2]{KrannRW}, the action of the generator of $C_2$ on $\pi_{4n-1-i}(\frac{\Top(2n+1)}{O(2n+1)})$ corresponds under this isomorphism to \emph{minus} the involution $(BI)_*$ considered in the proof of Lemma \ref{lem:doubling-map-disc} above. 
By \cite[Corollary 8.4]{KrannRW}, we therefore have, for \emph{even} $n$, that
\[
\pi_{4n-1}(\frac{\Top(2n+1)}{O(2n+1)}) \otimes \bQ= \bQ^+,
\]
\[
\pi_{4n-2}(\frac{\Top(2n+1)}{O(2n+1)}) \otimes \bQ= 0
\]
and 
\[
\pi_{4n-3}(\frac{\Top(2n+1)}{O(2n+1)}) \otimes \bQ= \bQ^-
\]
as $C_2$-modules. Since $H^2 (\rp^2;\bQ^+)=H^0 (\rp^2;\bQ^-)=0$, we obtain \eqref{eqn:cohomology-rp2}, which concludes the proof.
\end{proof}

\bibliographystyle{plain}
\bibliography{diffalgK}

\end{document}